\documentclass[a4paper,12pt]{article}

\makeatletter
\@addtoreset{footnote}{page}
\makeatother

\usepackage{style-enumitem}

\usepackage{amsthm}
\usepackage{amsmath,amssymb,latexsym,amsfonts,mathrsfs}

\newcommand{\eps}{\ensuremath{\varepsilon}}

\renewcommand{\tilde}{\widetilde}

\newcommand{\bN}{\ensuremath{\mathbb{N}}}

\newcommand{\bP}{\ensuremath{\mathbb{P}}}

\newcommand{\bR}{\ensuremath{\mathbb{R}}}

\newcommand{\cL}{\ensuremath{\mathcal{L}}}

\newcommand{\cP}{\ensuremath{\mathcal{P}}}

\theoremstyle{plain}
\newtheorem{Thm}{Theorem}[section]

\newtheorem{Lem}[Thm]{Lemma}
\newtheorem{Prop}[Thm]{Proposition}

\theoremstyle{definition}

\newtheorem{Rem}[Thm]{Remark}

\numberwithin{equation}{section}

\makeatletter
\renewcommand\section{\@startsection {section}{1}{\z@}%
	{-3.5ex \@plus -1ex \@minus -.2ex}%
	{2.3ex \@plus.2ex}%
	{\normalfont\large\bf}}
\makeatother

\makeatletter
\renewcommand\subsection{\@startsection {subsection}{1}{\z@}%
	{-3.5ex \@plus -1ex \@minus -.2ex}%
	{2.3ex \@plus.2ex}%
	{\normalfont\normalsize\bf}}
\makeatother

\usepackage{color}

\begin{document}
	
	\begin{center}
		{\Large \bf 
			A unifying approach to non-minimal
			quasi-stationary distributions for one-dimensional
			diffusions}
	\end{center}
	\begin{center}
		Kosuke Yamato\footnote{email: kyamato@math.kyoto-u.ac.jp} (Kyoto University)
	\end{center}
	\begin{center}
		{\small \today}
	\end{center}

	\begin{abstract}
		Convergence to non-minimal quasi-stationary distributions for one-dimensional diffusions is studied.
		We give a method of reducing the convergence to the tail behavior of the lifetime via a property which we call the first hitting uniqueness.
		We apply the results to Kummer diffusions with negative drifts 
		and give a class of initial distributions converging to each non-minimal quasi-stationary distribution.
	\end{abstract}

	\section{Introduction}
	Let us consider a one-dimensional diffusion $X = (X_t)_{t \geq 0}$ on $I = [0,b) \ \text{or} \ [0,b] \ (0 < b \leq \infty)$ killed nowhere and stopped upon hitting $0$ and let $T_0$ denote its first hitting time of $0$. 
	A probability distribution $\nu$ on $I \setminus \{0\}$ is called a {\it quasi-stationary distribution} of $X$ when the distribution of $X_t$ with the initial distribution $\nu$ conditioned to be away from $0$ until time $t$ is time-invariant, that is, the following holds:
	\begin{align}
		\bP_\nu[X_t \in dx \mid T_0 > t] = \nu(dx) \quad (t > 0), \label{}
	\end{align}
	where $\bP_{\nu}$ denotes the underlying probability measure of $X$ with its  initial distribution $\nu$.
	For a choice of a quasi-stationary distribution $\nu$, we study a sufficient condition on an initial distribution $\mu$ such that
	\begin{align}
	\mu_{t}(dx) := \bP_{\mu}[X_t \in dx \mid T_0 > t]  \xrightarrow[t \to \infty]{} \nu(dx), \label{eq35}
	\end{align} 
	where the convergence is the weak convergence of probability distributions.
	In the case where $\mu$ is compactly supported, the convergence \eqref{eq35} has been studied by many authors (e.g., Hening and Kolb \cite{Hening}, Kolb and Steinsaltz \cite{Kolb}, Littin \cite{Littin} and Mandl \cite{Mandl}), and it has been shown that under very general conditions the convergence \eqref{eq35} holds and the limit distribution $\nu$ does not depend on the choice of a compactly supported $\mu$. The limit measure $\nu$ is sometimes called Yaglom limit or the {\it minimal quasi-stationary distribution}. 
	On the other hand, for some diffusions there exists infinitely many quasi-stationary distributions. Although it is a natural problem to consider for what initial distributions the convergence \eqref{eq35} holds for each quasi-stationary distribution $\nu$, there are very few studies considering this problem for non-minimal quasi-stationary distributions.
	The author only knows two papers: Lladser and San Mart\'{\i}n \cite{DoAofOU} and Martinez, Picco and San Martin \cite{DoAofBM}, whose results we generalize in the present paper.
	
	In the present paper, we give a method of reducing the convergence \eqref{eq35} to the tail behavior of $T_0$. 
	For a class $\cP$ of initial distributions, we say that the {\it first hitting uniqueness} holds on $\cP$ if 
	\begin{align}
		\text{the map} \quad \cP \ni \mu \longmapsto \bP_{\mu}[T_0 \in dt] \quad 
		\text{is injective.} \label{}
	\end{align}
	As the class $\cP$, we shall take
	\begin{align}
	\cP_{\mathrm{exp}} = \{ \mu \in \cP(I) \mid \bP_{\mu}[T_0 \in dt] = \lambda \mathrm{e}^{-\lambda t}dt \quad (\lambda > 0) \}, \label{}
	\end{align}
	the set of initial distributions with exponential hitting probabilities,
	where $\cP(I)$ denotes the set of probability distributions on $I$.
	We refer to Rogers \cite{RogersFPP} as a general study of the first hitting uniqueness.
	
	One of our main results is a general result to reduce the convergence \eqref{eq35} to the tail behavior of $T_0$, provided that the first hitting uniqueness holds on $\cP_{\mathrm{exp}}$:
	\begin{Thm}\label{main-theorem-03}
		Let $X$ be a $\frac{d}{dm}\frac{d}{ds}$-diffusion on $[0,b) \ (0 < b  \leq \infty)$ and set
		\begin{align}
			\mu_{t}(dx) = \bP_{\mu}[X_t \in dx \mid T_0 > t]. \label{}
		\end{align}
		Assume the first hitting uniqueness holds on $\cP_{\mathrm{exp}}$ and
		\begin{align}
			\bP_{\nu_\lambda}[T_0 \in dt] = \lambda \mathrm{e}^{-\lambda t}dt \quad \text{for some}
			\ \lambda > 0 \ \text{and some} \ \nu_{\lambda} \in \cP(0,b).
		\end{align}
		Then for $\mu \in \cP[0,b)$ and $\lambda > 0$, the following are equivalent:
		\begin{enumerate}
			\item $
			\lim_{t \to \infty}\frac{\bP_{\mu}[T_0 > t + s]}{\bP_{\mu}[T_0 > t]} = \mathrm{e}^{-\lambda s} \ (s > 0)$.
			\item $\bP_{\mu_{t}}[T_0 \in ds] \xrightarrow[t \to \infty]{} \lambda \mathrm{e}^{-\lambda s}ds$.
			\item $\mu_t \xrightarrow
			[t \to \infty]{} \nu_\lambda$.
		\end{enumerate}
	\end{Thm}

	To study concrete sufficient conditions for the convergence \eqref{eq35},
	we introduce the class of {\it Kummer diffusions} with negative drifts.
	A Kummer diffusion $Y^{(0)} = Y^{(\alpha,\beta)}\ (\alpha > 0,\ \beta \in \bR)$ is a diffusion on $[0,\infty)$ stopped upon hitting $0$ whose local generator $\cL^{(0)} = \cL^{(\alpha,\beta)}$ on $(0,\infty)$ is
	\begin{align}
	\cL^{(0)} = \cL^{(\alpha,\beta)} &= x\frac{d^2}{dx^2} + (-\alpha + 1 - \beta x) \frac{d}{dx}. \label{eq5}
	\end{align}
	Note that the process $Y^{(0)} = Y^{(\alpha,\beta)}$ is also called a {\it radial Ornstein-Uhlenbeck process} in some literature (see e.g., \cite{BMhandbook} and \cite{Yor:Besselgen}).
	Write
	\begin{align}
	g_\gamma(x) := \bP_{x}[\mathrm{e}^{-\gamma T_0^{(0)}}] \quad (\gamma \geq 0), \label{eq75}
	\end{align}
	which is the Laplace transform of the first hitting time of $0$ for $Y^{(0)} = Y^{(\alpha,\beta)}$.
	Then $g_\gamma$ is a $\gamma$-eigenfunction for $\cL^{(0)}$, i.e., $\cL^{(0)} g_\gamma = \gamma g_\gamma$ (see e.g., \cite[p.292]{RogersWilliams2}).
	We define a Kummer diffusion with a negative drift $Y^{(\gamma)} = Y^{(\alpha,\beta,\gamma)} \ (\gamma \geq 0)$ as the $h$-transform of $Y^{(\alpha,\beta)}$ by the function $g_\gamma$, that is,
	the process $Y^{(\alpha,\beta,\gamma)}$ is a diffusion on $[0,\infty)$ stopped at $0$ whose local generator on $(0,\infty)$ is
	\begin{align}
	\cL^{(\gamma)} = \cL^{(\alpha,\beta,\gamma)} =  \frac{1}{g_\gamma}(\cL^{(0)} - \gamma) g_\gamma. \label{eq74}
	\end{align}
	
	If we write
	\begin{align}
	\tilde{Y}^{(\alpha,\beta,\gamma)} := \sqrt{2Y^{(\alpha,\beta,\gamma)}}, \label{}
	\end{align}
	then the local generator $\tilde{\cL}^{(\alpha,\beta,\gamma)}$ of $\tilde{Y}^{(\alpha,\beta,\gamma)}$ on $(0,\infty)$ is given as
	\begin{align}
	\tilde{\cL}^{(\alpha,\beta,\gamma)} 
	= \frac{1}{2}\frac{d^2}{dx^2} + \left( \frac{1-2\alpha}{2x} - \frac{\beta x}{2}  + \frac{\tilde{g}_\gamma'}{\tilde{g}_\gamma} \right)\frac{d}{dx}, \label{eq47}
	\end{align}
	where $\tilde{g}_\gamma(x) = \tilde{\bP}_x[\mathrm{e}^{-\gamma \tilde{T}_0}]$
	denotes the Laplace transform of the first hitting time of $0$ for $\tilde{Y}^{(0)}$ starting from $x$.
	When $\alpha = 1/2$ and $\gamma = 0$, the process $\tilde{Y}^{(1/2,\beta,0)}$ is the Ornstein-Uhlenbeck process and, when $\beta = 0$, the process $\tilde{Y}^{(\alpha,0,\gamma)}$ is the Bessel process with a negative drift (see e.g., \cite{Yor:Besselgen}).
	
	In Theorem \ref{QSD-char01}, we will see that for $\frac{d}{dm}\frac{d}{ds}$-diffusions satisfying a certain condition, existence of non-minimal quasi-stationary distributions is characterized by $\lambda_0 > 0$, where $\lambda_0 \ (\geq 0)$ denotes the spectral bottom, i.e., the bottom of $L^2$-spectrum of $-\frac{d}{dm}\frac{d}{ds}$ with Dirichlet boundary condition. 
	When non-minimal quasi-stationary distributions exist, they are parametrized by $\lambda \in (0,\lambda_0]$ and we denote the quasi-stationary distribution by $\nu_{\lambda}$.
	Note that there is an order such that $\nu_{\lambda'} \preceq \nu_{\lambda}$ for $0 < \lambda \leq \lambda' \leq \lambda_0$, so that $\nu_{\lambda_0}$ is the minimal quasi-stationary distribution
	(see Section \ref{section:preliminary} for the details).
	
	We classify $Y^{(\gamma)} = Y^{(\alpha,\beta,\gamma)} \ (\alpha > 0, \ \beta \in \bR, \ \gamma \geq 0)$ into the following five cases by $\beta$ and $\gamma$:
		
	\begin{align}
		\begin{aligned}
		&\text{Case 1:} & \beta = 0, \quad \gamma > 0. \\
		&\text{Case 2:} & \beta > 0, \quad \gamma \geq 0. \\
		&\text{Case 3:} & \beta < 0, \quad \gamma > 0. \\
		&\text{Case 1':} & \beta = 0, \quad \gamma = 0. \\
		&\text{Case 3':} & \beta < 0, \quad \gamma = 0. 
		\end{aligned} \label{cases}
	\end{align}
	
	We will see in Proposition \ref{NSCforQSDforKum} that the process $Y^{(\gamma)} = Y^{(\alpha,\beta,\gamma)} \ (\alpha > 0, \ \beta \in \bR, \ \gamma \geq 0)$ has infinitely many quasi-stationary distributions if and only if one of the Case 1-3 holds.
	
	The following is another main result of the present paper, where $L^1(I,\nu)$ denotes the set of integrable functions on $I$ w.r.t.\ the measure $\nu$:
	\begin{Thm} \label{thm:doaofKum}
		Let $X = Y^{(\gamma)} = Y^{(\alpha,\beta,\gamma)} \ (\alpha > 0, \ \beta \in \bR, \ \gamma \geq 0)$ satisfying one of the Case 1-3 in \eqref{cases} and let $\mu \in \cP(0,\infty)$.
		Then the following holds:
		\begin{enumerate}
			\item If the Case 1 holds and $\mu(dx) = \rho(x)dx$ for some $\rho \in L^1((0,\infty),dx)$ and 
			\begin{align}
			\log \rho(x) \sim (\delta - 2\sqrt{\gamma})\sqrt{x} \quad (x \to \infty) \label{eq52}
			\end{align}
			for some $0 < \delta < 2\sqrt{\gamma}$, then it holds
			\begin{align}
			\mu_{t} \xrightarrow[t \to \infty]{} \nu_\lambda \label{eq53}
			\end{align}
			with $\lambda = \gamma - \delta^2/4 \in (0,\lambda^{(\gamma)}_0)$, where $\lambda_0^{(\gamma)} = \gamma > 0$ is the spectral bottom.
			\item If the Case 2 holds and
			\begin{align}
			\mu(x,\infty) \sim x^{-\alpha - \gamma / \beta + \delta} \ell(x) \quad (x \to \infty) \label{eq61}
			\end{align}
			for some $0 < \delta <\alpha +  \gamma / \beta$ and some slowly varying function $\ell$ at $\infty$,
			then it holds
			\begin{align}
			\mu_{t} \xrightarrow[t \to \infty]{} \nu_\lambda \label{}
			\end{align}
			with $\lambda = \beta(\alpha - \delta) + \gamma \in (0,\lambda^{(\gamma)}_0)$, where $\lambda^{(\gamma)}_0 = \alpha\beta + \gamma > 0$ is the spectral bottom.
			\item If the Case 3 holds and
			\begin{align}
				\mu(x,\infty) \sim x^{-1 + \gamma / \beta + \delta}\ell(x) \quad (x \to \infty) \label{}
			\end{align}
			for some $0 < \delta < 1 - \gamma / \beta$ and some slowly varying function $\ell$ at $\infty$. then it holds
			\begin{align}
				\mu_{t} \xrightarrow[t \to \infty]{} \nu_\lambda \label{}
			\end{align}
			with $\lambda = -\beta(1 - \delta) + \gamma \in (0,\lambda^{(\gamma)}_0)$, where $\lambda^{(\gamma)}_0 = -\beta + \gamma > 0$ is the spectral bottom.

		\end{enumerate}
	\end{Thm}

	We will compare Theorem \ref{thm:doaofKum} with previous studies in Remarks \ref{rem:thm1} and \ref{rem:thm2}.

	\subsection*{Previous studies}\label{section:prev}
	
	We briefly review several previous studies on quasi-stationary distributions for one-dimensional diffusions.
	
	A first remarkable result on quasi-stationary distributions for one-dimensional diffusions was given by Mandl \cite{Mandl}. He treated the case the right boundary is natural and gave a sufficient condition for the convergence to the minimal quasi-stationary distributions.
	His condition has been weakened by many authors e.g., \cite{CMM}, \cite{Hening}, \cite{Kolb} and \cite{Ratesofdecay}. Under certain weak assumptions it is shown that all compactly supported initial distributions imply convergence to the minimal quasi-stationary distribution.
	
	The case where the right boundary is entrance has also been widely studied.
	Cattiaux et al.\ \cite{QSDpopulation} and Littin \cite{Littin} showed that in this case, there exists a unique quasi-stationary distribution and all compactly supported initial distributions are attracted to the unique quasi-stationary distribution.
	Takeda \cite{Takeda:QSD} generalized their results to symmetric Markov processes 
	with the {\it tightness property}.
	
	Let us come back to the case where the right boundary is natural.	
	We have then non-minimal quasi-stationary distributions.
	Firstly, Martinez, Picco and San Martin \cite{DoAofBM} studied Brownian motions with negative drifts and showed convergence to non-minimal quasi-stationary distributions under the assumptions on tail behavior of the initial distribution:
	\begin{Thm}[{\cite[Theorem 1.1]{DoAofBM}}] \label{prevStu:BM}
		Let $B_t$ be a standard Brownian motion and let $\alpha > 0$ and consider the process
		\begin{align}
			X_t = B_t - \alpha t. \label{}
		\end{align}
		For an initial distribution $\mu$ on $(0,\infty)$ assume
		$\mu(dx) = \rho(x)dx$ for some $\rho \in L^1((0,\infty),dx)$ satisfying
		\begin{align}
			\log \rho(x) \sim -(\alpha - \delta) x \quad (x \to \infty)
		\end{align}
		for some $\delta \in (0,\alpha)$.
		Then it holds
		\begin{align}
			\bP_{\mu}[X_t \in dx \mid  T_0 > t ] \xrightarrow[t \to \infty]{} \nu_{\lambda}(dx), \label{}
		\end{align}
		with
		\begin{align}
			\lambda = (\alpha^2 - \delta^2)/2 \quad \text{and}\quad 
			\nu_{\lambda}(dx) = C_\lambda \mathrm{e}^{-\alpha x}\sinh (x\sqrt{\alpha^2 - 2\lambda})dx \label{}
		\end{align}
		for the normalizing constant $C_\lambda$.
	\end{Thm}
	\begin{Rem}\label{rem:thm1}
		When $\alpha = 1/2, \beta = 0$ and $\gamma > 0$, the process $\sqrt{2 Y^{(1/2,0,\gamma)}}$ is a Brownian motion with a negative drift $-\sqrt{2\gamma} t$. Hence this theorem is generalized by (i) of Theorem \ref{thm:doaofKum}.
	\end{Rem}

	Secondly, Lladser and San Martin \cite{DoAofOU} studied Ornstein-Uhlenbeck processes:  
	\begin{Thm}[{\cite[Theorem 1.1]{DoAofOU}}]\label{prevStu:OU}
		Let $\alpha > 0$. Let $X$ be the solution of the following SDE:
		\begin{align}
		dX_t = dB_t - \alpha X_tdt, \label{}
		\end{align}
		where $B$ is a standard Brownian motion.
		For an initial distribution $\mu$ on $(0,\infty)$ assume
		$\mu(dx) = \rho(x)dx$ for some $\rho \in L^1((0,\infty),dx)$ satisfying
		\begin{align}
		\rho(x) \sim x^{-2 + \delta}\ell(x) \quad (x \to \infty)
		\end{align}
		for some $\delta \in (0,1)$ and a slowly varying function $\ell$ at $\infty$.
		Then it holds
		\begin{align}
		\bP_{\mu}[X_t \in dx \mid  T_0 > t ] \xrightarrow[t \to \infty]{} \nu_{\lambda}(dx) \label{}
		\end{align}
		with 
		\begin{align}
		\lambda = \alpha(1 - \delta) \quad \text{and}\quad 
		\nu_{\lambda}(dx) = C_\lambda \psi_{-\lambda}(x)\mathrm{e}^{-\alpha x^2}dx \label{}
		\end{align}
		for the normalizing constant $C_\lambda$, where $u = \psi_{-\lambda}$ denotes the unique solution for the following differential equation:
		\begin{align}
		\frac{1}{2}\frac{d^2}{dx^2}u - \alpha x\frac{d}{dx}u = -\lambda u, \quad 
		\lim_{x \to +0}u(x) = 0,\quad 
		\lim_{x \to +0}\frac{d}{dx}u(x) = 1 \quad (x \in (0,\infty)). \label{}	
		\end{align}   
	\end{Thm}

	\begin{Rem}\label{rem:thm2}
		In Theorem \ref{thm:doaofKum} (ii), if $\mu(dx) = \rho(x) dx$ for $\rho \in  L^1((0,\infty),dx)$ and
		\begin{align}
		\rho (x) \sim x^{-\alpha - \gamma /\beta + \delta - 1}\ell(x) \quad (x \to \infty), \label{}
		\end{align}
		for a slowly varying function $\ell$,
		then \eqref{eq61} holds from Karamata's theorem \cite[Proposition 1.5.8]{Regularvariation}.
		Hence (ii) of Theorem \ref{thm:doaofKum} is an extension of \cite[Theorem 1.1]{DoAofOU}.
	\end{Rem}

	\subsection*{Outline of the paper}
	The remainder of the present paper is organized as follows. In Section \ref{section:preliminary}, we recall several known results on one-dimensional diffusions, the quasi-stationary distributions and the spectral theory for second-order ordinary differential operators.
	In Section \ref{section:convtoqsdgeneral}, we show Theorem \ref{main-theorem-03}, a general result for convergence to quasi-stationary distributions.
	In Section \ref{section:examplesofFHP}, we give the hitting density of Kummer diffusions with negative drifts.
	In Section \ref{section:convtoqsd}, we show Theorem \ref{thm:doaofKum}, which gives a sufficient condition for convergence to non-minimal quasi-stationary distributions  for the class of Kummer diffusions with negative drifts.
	
\subsection*{Acknowledgement}
	The author would like to thank Kouji Yano, who read an early draft of this paper and gave him valuable comments. Thanks to him, the present paper was significantly improved. The author also would like to thank Toshiro Watanabe, who suggested the counterexample given in Remark \ref{counterexample}.

	\section{Preliminary}\label{section:preliminary}
	
	\subsection{Feller's canonical form of second order differential operators}
	
	Let $(X,\bP_x)_{x \in I}$ be a one-dimensional diffusion on $I = [0,b) \ \text{or} \ [0,b] \quad (0 < b \leq \infty)$, that is, the process $X$ is a time-homogeneous strong Markov process on $I$ which has a continuous path up to its lifetime.
	Throughout this paper, we always assume
	\begin{align}
	\bP_x[T_y < \infty] > 0 \quad (x \in I \setminus \{0\}, \ y \in [0,b)), \label{eq37}
	\end{align} 
	where $T_y$ denotes the first hitting time of $y$ and, assume the point $0$  a trap;
	\begin{align}
		X_t = 0 \quad \text{for} \ t \geq T_0. \label{}
	\end{align}
	
	Let us recall Feller's classification of the boundaries (see e.g., It\^o \cite{Ito_essentials}).
	There exist a Radon measure $m$ on $I \setminus \{0\}$ with full support and a strictly increasing continuous function $s$ on $(0,b)$ such that the local generator $\cL$ on $(0,b)$ is represented by
	\begin{align}
	\cL = \frac{d}{dm}\frac{d}{ds}. \label{}
	\end{align}
	We call $m$ {\it the speed measure} and $s$ {\it the scale function} of $X$ and we say $X$ is a $\frac{d}{dm}\frac{d}{ds}$-diffusion.
	Let $c = 0$ or $b$ and take $d \in (0,b)$. Set
	\begin{align}
		I(c) = \int_{c}^{d}ds(x)\int_{c}^{x}dm(y), \quad
		J(c) = \int_{c}^{d}dm(x)\int_{c}^{x}ds(y). \label{}
	\end{align}
	The boundary $c$ is classified as follows:
	 \begin{align}
	 \text{The boundary}\	c \ \text{is} \quad
	 	\left\{
	 	\begin{aligned}
	 		&\text{regular} &&\text{when}\  I(c) < \infty, \ J(c) < \infty. \\
	 		&\text{exit} &&\text{when} \ I(c) = \infty, \ J(c) < \infty. \\
	 		&\text{entrance} &&\text{when} \ I(c) < \infty, \ J(c) = \infty. \\
	 		&\text{natural} &&\text{when} \ I(c) = \infty, \ J(c) = \infty.
	 	\end{aligned}
	 	\right.
	 \end{align}
	Since $\bP_{x}[T_0 < \infty] > 0$ for every $x > 0$, the boundary $0$ is necessarily regular or exit, equivalently $J(0) < \infty$. Note that in this case $s(0): = \lim_{x \to +0}s(x) > -\infty$ holds. 
	We also assume that the boundary $b$ is not exit and that the boundary $b$ is reflecting when it is regular.
	
	Let us consider a diffusion on $I$ whose local generator $\cL$ on $(0,b)$ is
	\begin{align}
		\cL = a(x)\frac{d^2}{dx^2} + c(x) \frac{d}{dx} \quad (x \in (0,b))
	\end{align}
	for functions $a$ and $c$. Assume $a(x) > 0 \ (x \in (0,b))$. Then $\cL = \frac{d}{dm}\frac{d}{ds}$, where
	\begin{align}
		dm(x) = \frac{1}{a(x)}\exp \left(\int_{d}^{x}\frac{c(y)}{a(y)}dy \right)dx, \quad 
		ds(x) = \exp \left(-\int_{d}^{x}\frac{c(y)}{a(y)}dy\right)dx \quad 
 	\end{align}
	for arbitrary taken $d \in (0,b)$.
	

%

	\subsection{Quasi-stationary distributions} \label{subsection:qsd}
	
	Let us summarize known results on quasi-stationary distributions for one-dimensional diffusions and give a necessary and sufficient condition for existence of quasi-stationary distributions.
	Let $X$ be a $\frac{d}{dm}\frac{d}{ds}$-diffusion on $I = [0,b)$ or $[0,b] \ (0 < b \leq \infty)$. 
	We define a function $u = \psi_{\lambda}$ as the unique solution of the following equation:
	\begin{align}
	\frac{d}{dm}\frac{d}{ds}u(x) = \lambda u(x), \quad 
	\lim_{x \to +0}u(x) = 0,\quad 
	\lim_{x \to +0}\frac{d}{ds}u(x) = 1 \quad (x \in (0,b), \lambda \in \bR). \label{eq63}	
	\end{align}   
	Note that from the assumption that the boundary $0$ is regular or exit, the function $\psi_{\lambda}$ always exists.
	The operator $L = -\frac{d}{dm}\frac{d}{ds}$ defines a non-negative definite self-adjoint operator on $L^2(I,dm) := \{ f:I \to \bR \mid \int_{I}|f|^2dm < \infty \}$.
	Here we assume the Dirichlet boundary condition at $0$ and the Neumann boundary condition at $b$ if the boundary $b$ is regular. We denote the infimum of the spectrum of $L$ by $\lambda_0 \geq 0$.
	
	Let us consider the case where the boundary $b$ is not natural. It is then known that there is a unique quasi-stationary distribution (noting that Takeda \cite{Takeda:QSD} showed the corresponding result for general Markov processes with the tightness property): 
	\begin{Prop}[{see e.g., \cite[Lemma 2.2, Theorem 4.1]{Littin}}]
		Assume the boundary $b$ is not natural.
		Then it holds $\lambda_0 > 0$ and the function $\psi_{-\lambda_0}$ is strictly positive and integrable w.r.t.\ $dm$ and, there is a unique quasi-stationary distribution given as
		\begin{align}
			\nu_{\lambda_0}(dx) = \lambda \psi_{-\lambda_0}(x)dm(x), \quad \bP_{\nu_{\lambda_0}}[T_0 \in dt] = \lambda_0 \mathrm{e}^{-\lambda_0 t}dt. \label{}
		\end{align}
		Moreover, for every probability distribution $\mu$ on $(0,b)$ with a compact support,
		it holds
		\begin{align}
			\mu_{t} \xrightarrow[t \to \infty]{} \nu_{\lambda_0}. \label{}		
		\end{align}
	\end{Prop}
	
	We now assume the boundary $b$ is natural.
	We now have
	\begin{align}
		\bP_{x}[T_b < \infty] = 0 \quad (x \in (0,b)), \label{eq40}
	\end{align} 
	and
	\begin{align}
		\frac{s(x) - s(0)}{s(M) - s(0)} = \bP_{x}[T_M < T_0] \quad (0 < x < M < b) \label{}
	\end{align}
	(see e.g., It\^o \cite{Ito_essentials}).
	Taking limit $M \to b$, we have from \eqref{eq40}
	\begin{align}
		\frac{s(x) - s(0)}{s(b) - s(0)} = \bP_x[T_0 = \infty]. \label{}
	\end{align} 
	Hence it follows
	\begin{align}
		\bP_x[T_0 < \infty] = 1 \quad \text{for some / any} \ x > 0 \quad  \Leftrightarrow \quad s(b) = \infty. \label{} 
	\end{align}
	If $\nu$ is a quasi-stationary distribution, the distribution $\bP_{\nu}[T_0 \in dt]$ is exponentially distributed because $\bP_{\nu}[T_0 > t+s \mid T_0 > t] = \bP_{\nu}[X_{t+s} > 0 \mid T_0 > t] = \bP_{\nu}[X_s > 0] = \bP_{\nu}[T_0 > s]$. Then by \eqref{eq37} it holds $\bP_{\nu}[T_0 = \infty] < 1$ and therefore $\bP_{\nu}[T_0 = \infty] = 0$, which implies $s(b) = \infty$.
	We recall the following good properties for the function $\psi_{\lambda}$:
	\begin{Prop}[{\cite[Lemma 6.18]{Quasi-stationary_distributions}}]
		Suppose the boundary $b$ is natural and $s(b) = \infty$.
		Then for $\lambda > 0$ the following hold:
		\begin{enumerate}
			\item For $0 < \lambda \leq \lambda_0$, the function $\psi_{-\lambda}$ is strictly positive on $I$ and 
			\begin{align}
			1 = \lambda \int_{0}^{b}\psi_{-\lambda}(x)dm(x). \label{eq78}
			\end{align}
			\item For $\lambda > \lambda_0$, the function $\psi_{-\lambda}$ change signs on $I$.
		\end{enumerate} 
	\end{Prop}	
	
	Now we state a necessary and sufficient condition for existence of non-minimal quasi-stationary distributions without proof:
	\begin{Thm}[{\cite[Theorem 6.34]{Quasi-stationary_distributions} and \cite[Theorem 3, Appendix I]{KotaniWatanabe}}] \label{QSD-char01}
		Suppose the boundary $b$ is natural.
		Then a non-minimal quasi-stationary distribution exists if and only if
		\begin{align}
			\lambda_0 > 0 \quad \text{and}\quad s(b) = \infty. \label{eq39}
		\end{align}
		This condition is equivalent to 
		\begin{align}
		m(d,b) < \infty \quad \text{for some} \ d \in (0,b) \quad \text{and} \quad \limsup_{x \to b}s(x)m(x,b) < \infty. \label{eq42}			
		\end{align} 
		In this case, a probability measure $\nu$ is a quasi-stationary distribution if and only if 
		\begin{align}
			\nu(dx) = \lambda \psi_{-\lambda}(x)dm(x) = :\nu_\lambda(dx), \quad \bP_{\nu_{\lambda}}[T_0 \in dt] = \lambda \mathrm{e}^{-\lambda t}dt \quad \text{for some} \ 0 < \lambda \leq \lambda_0 . \label{eq73}
		\end{align}
	\end{Thm}
	Here we note that as \cite{Quasi-stationary_distributions} only dealt with the case the boundary $0$ is regular, the proof also works in the case the boundary $0$ is exit.

	For probability distributions on $(0,b)$, we introduce a partial order. For $\mu_1,\ \mu_2 \in \cP(0,\infty)$, we define $\mu_1 \preceq \mu_2$ by
	\begin{align}
		\mu_2(0,x] \leq \mu_1(0,x] \quad (x > 0). \label{}
	\end{align}
	This order gives a total order for quasi-stationary distributions and, as the following proposition says, the distribution $\nu_{\lambda_0}$ gives the minimal element. This is why we call it the minimal quasi-stationary distribution.
	\begin{Prop}
		Suppose the boundary $b$ is natural and \eqref{eq39} holds.
		Then it holds
		\begin{align}
			\nu_{\lambda} \preceq \nu_{\lambda'} \quad (0 < \lambda' \leq \lambda \leq \lambda_0). \label{}
		\end{align}
		In particular, the distribution $\nu_{\lambda_0}$ is the minimal one in this order.
	\end{Prop}  
	\begin{proof}
		From \eqref{eq63}, it holds
		\begin{align}
			\psi_{-\lambda}(x) = s(x) - \lambda \int_{0}^{x}ds(y)\int_{0}^{y}\psi_{-\lambda}(z)dm(z) \quad (x > 0, \lambda \in \bR). \label{} 
		\end{align}
		Hence it follows
		\begin{align}
			\nu_\lambda(0,x] = \lambda\int_{0}^{x}\psi_{-\lambda}(y)dm(y) =  1 - \psi_{-\lambda}^+(x) \quad (x > 0, 0 < \lambda \leq \lambda_0), \label{eq79}	
		\end{align}
		where $\psi_{-\lambda}^+(x)$ is the right-derivative of $\psi_{-\lambda}$ w.r.t.\ the scale function:
		\begin{align}
			\psi_{-\lambda}^+(x) := \lim_{h \to +0}\frac{\psi_{-\lambda}(x + h) - \psi_{-\lambda}(x)}{s(x+h) - s(x)}. \label{}
		\end{align}
		
		Let $0 < \lambda' \leq \lambda \leq \lambda_0$.
		From \eqref{eq79} we have
		\begin{align}
			\psi_{-\lambda}^+(x) \leq \psi_{-\lambda'}^+(x) \quad (x > 0) \label{}
		\end{align}
		by a similar argument in \cite[Lemma 6.11]{Quasi-stationary_distributions},
		which yields $\nu_{\lambda} \preceq \nu_{\lambda'}$.
	\end{proof}

	\subsection{Spectral theory for second-order differential operators}
	
	Let us briefly review several results on the spectral theory of second-order differential operators. For the details, see e.g., Coddington and Levinson \cite{Coddington} and Kotani \cite{Kotani:singularleft}.
	 
	Set $I = (0,b) \ (0 < b \leq \infty)$.
	Let $dm$ be a Radon measure on $I$ with full support and let $s:I \to (-\infty,\infty)$ be a strictly-increasing continuous function.
	We assume that the boundary $0$ is regular or exit, i.e.\ $\int_{0}^{d}dm(x)\int_{0}^{x}ds(y) < \infty$ for some $0 <d < b$ and assume the boundary $b$ is natural, i.e., $\int_{d}^{b}dm(x)\int_{x}^{b}ds(y) =  \infty$ and $\int_{d}^{b}ds(x)\int_{x}^{b}dm(y) =  \infty$ for some $0 <d < b$. 
	Let $u = \psi_{\lambda}$ be defined by \eqref{eq63}. Set
	\begin{align}
		g_\lambda(x) = \psi_{\lambda}(x)\int_{x}^{b}\frac{ds(y)}{\psi_{\lambda}(y)^2} \quad (\lambda \geq 0). \label{eq55}
	\end{align}
	Then the function $u = g_\lambda$ is the unique, non-increasing solution for
	\begin{align}
		\frac{d}{dm}\frac{d}{ds} u = \lambda u, \quad \lim_{x \to +0}u(x) = 1. \label{}
	\end{align}
	Define the Green function
	\begin{align}
		G_\lambda(x,y) = G_\lambda(y,x) := \psi_{\lambda}(x)g_\lambda(y) \quad (0 \leq x \leq y < b, \ \lambda \geq 0). \label{}
	\end{align}
	Then there exists a unique Radon measure $\sigma$ on $[0,\infty)$, which we call the {\it spectral measure}, such that
	\begin{align}
		G_\lambda(x,y) =  \int_{0}^{\infty}\frac{\psi_{-\xi}(x)\psi_{-\xi}(y)}{\lambda + \xi}\sigma(d\xi) \label{}
	\end{align}
	and the transition density $p(t,x,y)$ w.r.t.\ $dm$ of $\frac{d}{dm}\frac{d}{ds}$-diffusion absorbed at $0$ is given as
	\begin{align}
		p(t,x,y) = \int_{0}^{\infty}\mathrm{e}^{-\lambda t}\psi_{-\lambda}(x)\psi_{-\lambda}(y)\sigma(d\lambda) \quad (t > 0, x,y \in I)
	\end{align}
	(see \cite{McKean:elementary} for the details).
	Note that, under the assumptions of Theorem \ref{QSD-char01}, the spectral measure has its support on $[\lambda_0,\infty)$.


	\section{Convergence to quasi-stationary distributions}\label{section:convtoqsdgeneral}
	
	Let $X$ be a $\frac{d}{dm}\frac{d}{ds}$-diffusion on $[0,b) \ (0 < b  \leq \infty)$ satisfying the condition of Theorem \ref{QSD-char01}.
	As a class of initial distributions we define 
	\begin{align}
		\cP_{\mathrm{exp}} := \{ \mu \in \cP[0,b) \mid \bP_{\mu}[T_0 \in dt] = \lambda \mathrm{e}^{-\lambda t}dt \quad (\lambda > 0) \}. \label{}
	\end{align}
	Provided that the first hitting uniqueness holds on $\cP_{\mathrm{exp}}$,
	an initial distribution $\mu \in \cP[0,b)$ satisfying 
	$\bP_{\mu}[T_0 \in dt] = \lambda \mathrm{e}^{-\lambda t}dt$ for some $0 < \lambda \leq \lambda_0$ is given as	$\mu = \nu_\lambda$.
	
	
	

	Let us prove Theorem \ref{main-theorem-03}:

	\begin{proof}[Proof of Theorem \ref{main-theorem-03}]
		From the Markov property, we have
		\begin{align}
		\bP_{\mu_{t}}[T_0 > s] = \frac{\bP_{\mu}[T_0 > t + s]}{\bP_{\mu}[T_0 > t]} \quad (t,s \geq 0). \label{eq26}
		\end{align}
		Now it is obvious that (i) and (ii) are equivalent.
		In addition, it is not difficult to see that (iii) implies (i).
		
		We show (ii) implies (iii).
		Since $\cP[0,b]$, the class of probability measures on the compactification $[0,b]$, is compact under the topology of weak convergence, we can take a sequence $\{t_n\}_n$ which diverges to $\infty$ such that 
		\begin{align}
		\mu_{t_n} \xrightarrow[n \to \infty]{} \nu \in \cP[0,b]. \label{eq16}
		\end{align}
		From (ii), we have
		\begin{align}
		\bP_{\mu_{t_n}}[T_0 \in ds] \xrightarrow[n \to \infty]{} \lambda \mathrm{e}^{-\lambda s}ds. \label{eq17}
		\end{align}
		On the other hand, for fixed $t > 0$ we have
		\begin{align}
		\bP_{\mu_{t_n}}[T_0 > t] = \int_{[0,b]}\bP_{x}[T_0 > t] \mu_{t_n}(dx), \label{}
		\end{align}
		where we understand
		\begin{align}
		\bP_{x}[T_0 > t] = 
		\begin{cases}
		0 & x = 0, \\
		1 & x = b. 
		\end{cases}
		\end{align}
		Note that since the boundary $b$ is natural, the function $x \mapsto \bP_x[T_0 > t]$ is continuous on $[0,b]$. From \eqref{eq16}, we obtain
		\begin{align}
		\lim_{n \to \infty}\bP_{\mu_{t_n}}[T_0 > t] = \int_{[0,b]}\bP_x[T_0 > t] \nu(dx). \label{} 
		\end{align}
		Then from \eqref{eq17}, it follows that
		\begin{align}
		\int_{[0,b]}\bP_x[T_0 > t] \nu(dx) = \mathrm{e}^{-\lambda t}. \label{eq36}
		\end{align}
		Since it holds that 
		\begin{align}
		\lim_{t \to 0}\bP_x[T_0 > t] = 1\{x > 0\}, \quad \lim_{t \to \infty}\bP_x[T_0 > t] = 1\{x = b \} \quad (x \in [0,b]), \label{}
		\end{align}
		we have from the dominated convergence theorem and \eqref{eq36} that $\nu\{0\} = \nu\{b\} = 0$.
		Therefore $\nu \in \cP(0,b)$ and $\bP_{\nu}[T_0 \in ds] = \lambda \mathrm{e}^{-\lambda s}ds$. Then since the first hitting uniqueness holds on $\cP_{\mathrm{exp}}$, we have $\nu = \nu_\lambda$. The limit distribution $\nu_\lambda$ does not depend on the choice of the sequence $\{t_n\}$ and therefore we obtain (iii).
	\end{proof}
	
	We give a sufficient condition for (i) of Theorem \ref{main-theorem-03}.
	\begin{Prop} \label{non-minimalQSD_density}
		Assume the hitting densities $f_x$ of $0$ exist, i.e., there exists a non-negative jointly measurable function $f_x(t)$ such that 
		\begin{align}
			\bP_{x}[T_0 \in dt] = f_x(t)dt \quad ( 0 < x < b,\ t > 0). \label{}
		\end{align}
		Let $\mu \in \cP(0,b)$ and assume the function 
		\begin{align}
		f_\mu(t) := \int_{0}^{\infty}f_x(t)\mu(dx) \quad (0 < x < b, \ t > 0) \label{}
		\end{align}
		is differentiable in $t > 0$ and
		\begin{align}
		-\lim_{t \to \infty}\frac{d}{dt}\log f_\mu(t) = \lambda \in (0,\lambda_0]. \label{eq25}
		\end{align}
		Then it holds
		\begin{align}
		\lim_{t \to \infty}\frac{\bP_{\mu}[T_0 > t+s]}{\bP_{\mu}[T_0 > t]} = \mathrm{e}^{-\lambda s} \quad (s > 0). \label{eq54}
		\end{align}
	\end{Prop}
	
	\begin{proof}
		Set $g(u) = f_\mu(\log u)$ for $u > 1$.
		From \eqref{eq25}, we have 
		\begin{align}
		\lim_{t \to \infty}\frac{tg'(t)}{g(t)} = \lim_{t \to \infty}\frac{\mathrm{e}^t g'(\mathrm{e}^t)}{g(\mathrm{e}^t)} = -\lambda. \label{}
		\end{align}
		Then from \cite[Theorem 2]{Lamperti}, the function $g$ varies regularly at $\infty$ with exponent $-\lambda$.
		From L'H\^opital's rule, we have for $u = \mathrm{e}^s > 1$
		\begin{align}
		\lim_{t \to \infty}\frac{\bP_{\mu}[T_0 > t +\log u]}{\bP_{\mu}[T_0 > t]} 
		= \lim_{t \to \infty}\frac{f_\mu(t+ \log u)}{f_\mu(t)} = \lim_{t \to \infty}\frac{g(\mathrm{e}^t u)}{g(\mathrm{e}^t)} = u^{-\lambda} = \mathrm{e}^{-\lambda s}. \label{}
		\end{align}		
	\end{proof}

	\begin{Rem} \label{counterexample}
		We may expect that Proposition \ref{non-minimalQSD_density} would be extended with \eqref{eq25} being replaced by 
		\begin{align}
			\log f_\mu(t) \sim -\lambda t \quad (t \to \infty), \label{eq56}
		\end{align}
		which is weaker than \eqref{eq25} by L'H\^opital's rule.
		In general, however, it does not hold.
		We give a counterexample which satisfies \eqref{eq56} but does not \eqref{eq54}. 
		Let us find a positive function $f$ of the form
		\begin{align}
			f(t) = \mathrm{e}^{(-\lambda + \eps(t))t} \label{}
		\end{align}
		with a function $\eps(t)$ varnishing at $\infty$ but not satisfying
		\begin{align}
			\frac{\int_{t+s}^{\infty}f(u)du}{\int_{t}^{\infty}f(u)du} \xrightarrow[t \to \infty]{} \mathrm{e}^{-\lambda s}  \quad (s > 0). \label{eq57}
		\end{align}
		By the change of variables, we can see that \eqref{eq57} is equivalent to that the function
		\begin{align}
			h(t) := \int_{t}^{\infty}u^{-\lambda - 1 + \eps(\log u) }du \label{eq58}
		\end{align}
		varies regularly with exponent $-\lambda $ at $\infty$.
		If the function $\eps$ is non-increasing, by the monotone density theorem \cite[Theorem 1.7.2]{Regularvariation} it is equivalent to the slow variation of 
		\begin{align}
			k(s) = s^{\eps(\log s)} \quad (s > 1). \label{}
		\end{align}
		We now set
		\begin{align}
			\eps(s) = 2^{-n} \quad (4^n < s \leq 4^{n+1}, n \in \bN), \label{}
		\end{align}
		and then the function $\eps$ vanishes at $\infty$ and
		\begin{align}
			\frac{k(\mathrm{e} \cdot \exp(4^n))}{k(\exp(4^n))} = \frac{\exp(2^{n} + 2^{-n})}{\exp(2^{n+1})}= \exp (-2^n + 2^{-n}) \xrightarrow[n \to \infty]{} 0. \label{}
		\end{align}
		So the function $k$ does not vary slowly.
	\end{Rem}

	We give a sufficient condition for existence of the hitting densities of $0$. 
	For this purpose, we need the following condition on decay of the spectral measure $\sigma$ of $-\frac{d}{dm}\frac{d}{ds}$:
	\begin{align}
	\mathrm{(S)} \quad \int_{0}^{\infty}\mathrm{e}^{-\lambda t} \sigma(d\lambda) < \infty \quad
	(t > 0). \label{condS}
	\end{align}
	A sufficient condition for $\mathrm{(S)}$ is as follows:
	\begin{Prop}[{\cite{report}}]\label{main-appendix}
		Let $m$ be a speed measure and $s$ be a scale function on $(0,b) \ (0 < b \leq \infty)$.
		Then if $|s(0)| < \infty$ and
		\begin{align}
		m(x,c] \leq C (s(x) - s(0))^{-\delta} \quad (0 < x < c) \label{} 
		\end{align}
		for some $C > 0$, $0 < c < b$ and $0 < \delta < 1$,
		the condition $\mathrm{(S)}$ holds.
	\end{Prop}
	The proof of Proposition \ref{main-appendix} is given in \cite{report}.
	The following result by Yano \cite{Yano:Excusionmeasure} gives existence and a spectral representation of the hitting densities.
	\begin{Prop}[{\cite[Proposition 2.1]{Yano:Excusionmeasure}}] \label{hitting_density}
		Assume $\mathrm{(S)}$ holds.
		Then for any $0 < x < b$ the distribution of $T_0$ under $\bP_{x}$ has a density $f_x(t)$ on $(0,\infty)$ w.r.t.\ Lebesgue measure, that is, the following hold:
		\begin{align}
		\bP_x[T_0 \in dt] = f_x(t)dt  \quad (0 < x < b, \ t > 0). \label{eq19}
		\end{align}
		The hitting densities have a spectral representation:
		\begin{align}
		f_x(t) = \int_{0}^{\infty}\mathrm{e}^{-\lambda t}\psi_{-\lambda}(x)\sigma(d\lambda) \quad (0 < x < b,\ t > 0). \label{eq20}
		\end{align}
		and have another representation:
		\begin{align}
		f_{x}(t) = \left.\frac{d}{ds(y)}p(t,x,y)\right|_{y=0} \quad (0 < x < b, \ t > 0). \label{}
		\end{align}
	\end{Prop}
	
	\section{Hitting densities of Kummer diffusions with negative drifts} \label{section:examplesofFHP}
	

	
	Let us give the hitting densities of Kummer diffusions with negative drifts.
	
	At first we give a speed measure and a scale function for Kummer diffusions with negative drifts.
	Fix $\alpha > 0$ and $\beta \in \bR$. From \eqref{eq5}, we have
	\begin{align}
	\cL^{(0)} = \cL^{(\alpha,\beta)} = x\frac{d^2}{dx^2} + (-\alpha + 1 - \beta x) \frac{d}{dx} = \frac{d}{dm^{(0)}}\frac{d}{ds^{(0)}}
	\end{align}
	with
	\begin{align}
	dm^{(0)}(x) := dm^{(\alpha,\beta)}(x) = x^{-\alpha}\mathrm{e}^{-\beta x}dx, \quad	ds^{(0)}(x) := ds^{(\alpha,\beta)}(x) = x^{\alpha - 1}\mathrm{e}^{\beta x}dx. \label{speedmeasure_scalefunction}
	\end{align}
	In addition, for $\gamma \geq 0$, we have 
	\begin{align}
	\cL^{(\gamma)} = \cL^{(\alpha,\beta,\gamma)} = x\frac{d^2}{dx^2} + \left(-\alpha + 1 - \beta x + \frac{xg_\gamma'(x)}{g_\gamma(x)}\right) \frac{d}{dx} = \frac{d}{dm^{(\gamma)}}\frac{d}{ds^{(\gamma)}} \label{}
	\end{align}
	with
	\begin{align}
	dm^{(\gamma)} = g_\gamma^2 dm^{(0)} \quad ds^{(\gamma)} = g_\gamma^{-2}ds^{(0)}, \label{h-transformdensity}
	\end{align}
	where $g_\gamma$ is the function given in \eqref{eq75}.
	Note that since $g_\gamma(0) = 1$, the classification of the boundary $0$ for $\cL^{(\gamma)}$ does not depend on $\gamma \geq 0$. The boundary $\infty$ for $\cL^{(\gamma)}$ is always natural, which we will see in Proposition \ref{classification_of_boundary}.
	We also have
	\begin{align}
		\cL^{(\gamma)} = \cL^{(0)} + \frac{x g_\gamma'}{g_\gamma}\frac{d}{dx} \label{}
	\end{align}
	and, by the obvious relation
	\begin{align}
	\tilde{g}_\gamma(x) = g_\gamma(x^2/2), \label{}
	\end{align}
	it follows
	\begin{align}
	\tilde{\cL}^{(\alpha,\beta,\gamma)} 
	= \tilde{\cL}^{(\alpha,\beta,0)} +  \frac{\tilde{g}_\gamma'}{\tilde{g}_\gamma}\frac{d}{dx}, \label{}
	\end{align}
	which implies \eqref{eq47}.
	
	We summarize several results on the hitting densities for Kummer diffusions with negative drifts.
	Note that from \eqref{speedmeasure_scalefunction} and Proposition \ref{main-appendix}, the condition $\mathrm{(S)}$ holds for $\frac{d}{dm^{(0)}}\frac{d}{ds^{(0)}}$. 
	
%
	
	\begin{Thm}\label{hitting-density-example}
		For the process $Y^{(\alpha,\beta,\gamma)} \ (\alpha > 0 , \ \beta \in \bR, \ \gamma \geq 0)$, the hitting densities $f^{(\gamma)}_{x}$ of $0$ and the spectral measure $\sigma^{(\gamma)}$ for $\cL^{(\gamma)}$ are given as
		\begin{align}
			f^{(\gamma)}_{x}(t) = \frac{\mathrm{e}^{-\gamma t}}{g_\gamma(x)}f^{(0)}_x(t) \quad (0 < x < \infty, \ t > 0) \label{eq77}
		\end{align}
		and
		\begin{align}
			\sigma^{(\gamma)}(d\lambda) = \sigma^{(0)}(d(\lambda - \gamma)) \label{eq80}
		\end{align}
		for
		\begin{align}
		f^{(0)}_{x}(t) = \left\{
		\begin{aligned}
		&\frac{1}{\Gamma (\alpha)}{x^{\alpha}t^{-\alpha - 1}}\mathrm{e}^{-x/t} &  ( \beta = 0), \\
		&\frac{x^{\alpha}\mathrm{e}^{\beta t}}{\Gamma(\alpha)}\left(\frac{ \beta \mathrm{e}^{-\beta t}}{1 - \mathrm{e}^{-\beta t}}\right)^{1+\alpha}\exp \left( \frac{-x\beta \mathrm{e}^{-\beta t}}{1 - \mathrm{e}^{-\beta t}} \right) & ( \beta \neq 0), \label{}
		\end{aligned}
		\right. \label{eq76}
		\end{align}
		and
		\begin{align}
			\sigma^{(0)}(d\lambda) = 
			\left\{
			\begin{aligned}
				&\beta^{\alpha+1}\sum_{n=0}^{\infty} \frac{(\alpha)_{n+1}}{n! \Gamma(\alpha)}\delta_{\beta(n + \alpha)}(d\lambda) & (\beta > 0), \\
				&\frac{1}{\Gamma(\alpha)^2}\lambda^{\alpha}d\lambda & (\beta = 0), \\
				&(-\beta)^{\alpha+1}\sum_{n=0}^{\infty} \frac{(\alpha)_{n+1}}{n! \Gamma(\alpha)}\delta_{(-\beta)(n + 1)}(d\lambda) & (\beta < 0),
			\end{aligned}
			\right. \label{eq81}
		\end{align}
		where $(a)_k \ (a \in \bR, \ k \in \bN)$ is a Pochhammer symbol
		\begin{align}
		(a)_k = a(a+1)\cdots (a + k - 1). \label{}
		\end{align}
		In particular, we have
		\begin{align}
			\lambda^{(\gamma)}_0 = \left\{ 
			\begin{aligned}
				&\alpha \beta + \gamma & (\beta > 0), \\
				&\gamma & (\beta = 0), \\
				&-\beta + \gamma & (\beta < 0).
			\end{aligned}
			\right. \label{}
		\end{align}
	\end{Thm}

	\begin{Rem}
		From e.g., \cite[Section 3.7]{Specialfunction}, we have
		\begin{align}
		g_{\gamma}(x) = \left\{
		\begin{aligned}
		&\frac{1}{2^{\alpha - 1} \Gamma(\alpha)}(2\sqrt{\gamma x})^{\alpha}K_\alpha(2\sqrt{\gamma x}) &  ( \beta = 0), \\
		&\frac{\Gamma(\alpha + \gamma / \beta)}{\Gamma (\alpha)}
		(\beta x)^{\alpha} U (\alpha + \gamma / \beta, \alpha + 1, \beta x) & ( \beta > 0), \label{} \\
		&\frac{\Gamma(1 - \gamma / \beta)}{\Gamma (\alpha)}
		(-\beta x)^{\alpha} \mathrm{e}^{\beta x}U (1 - \gamma /\beta, \alpha + 1 ; -\beta x) & ( \beta < 0), \label{} \\
		\end{aligned}
		\right. \label{eq62}
		\end{align}
		where $K_\alpha$ denotes the modified Bessel function of the second kind (see e.g., \cite[Section 3.1]{Specialfunction}) and $U$ denotes the Tricomi confluent hypergeometric function:
		\begin{align}
		U(a,b;x) = \frac{1}{\Gamma(a)}\int_{0}^{\infty}\mathrm{e}^{-sx}s^{a-1}(1+s)^{b-a-1}ds \quad (a > 0, \ b\in \bR, \ x > 0). \label{}
		\end{align}
		Note that
		\begin{align}
		K_\alpha(x) \sim 2^{\alpha -1}\Gamma(\alpha) x^{-\alpha}, \quad 
		U(a,b;x) \sim \frac{\Gamma(b - 1)}{\Gamma (a)}x^{-b+1} \quad (x \to + 0,\ a> 0, \ b > 1) \label{}
		\end{align}
		and
		\begin{align}
		K_\alpha(x) \sim \sqrt{\frac{\pi}{2x}}\mathrm{e}^{-x}, \quad 
		U(a,b,x) \sim x^{-a} \quad (x \to +\infty,\ a> 0) \label{eq50}
		\end{align}
		(see e.g., \cite[Section 3.14.1]{Specialfunction}).
	\end{Rem}

	We now prove Theorem \ref{hitting-density-example}.

	\begin{proof}[Proof of Theorem \ref{hitting-density-example}]
		At first, we show \eqref{eq77} and \eqref{eq76}. 
		We denote the transition probability of $Y^{(\gamma)} = Y^{(\alpha, \beta, \gamma)}$ by
		\begin{align}
			\bP_x[Y^{(\gamma)}_t \in dy] = p^{(\gamma)}(t,x,y)dm^{(\gamma)}(y)
		\end{align}
		Then it holds 
		\begin{align}
		p^{(\gamma)}(t,x,y) = \mathrm{e}^{-\gamma t}\frac{p^{(0)}(t,x,y)}{g_\gamma(x)g_\gamma(y)}, \label{}
		\end{align}
		(see e.g., \cite[p.172]{TakemuraTomisaki:htransform}),
		where we write $\bP_x$ for the underlying probability measure for $Y^{(\gamma)}$ starting from $x$.
		From \cite[Appendix 1]{BMhandbook}, the transition density $p^{(0)}(t,x,y)$ is given as
		\begin{align}
		p^{(0)}(t,x,y) = \left\{
		\begin{aligned}
		&\frac{1}{t}(xy)^{\alpha/2}\mathrm{e}^{-(x+y)/t}I_\alpha\left(\frac{2\sqrt{xy}}{t}\right) & (\beta = 0), \\
		&\frac{\beta \mathrm{e}^{-\alpha\beta t/2}}{1 - \mathrm{e}^{-\beta t}}(xy)^{\alpha/2}\exp \left(-\frac{(x + y)\beta \mathrm{e}^{-\beta t}}{1 - \mathrm{e}^{-\beta t}}\right) I_\alpha \left(\frac{2\sqrt{xy}\beta \mathrm{e}^{-\beta t / 2}}{1 - \mathrm{e}^{-\beta t}}\right) & (\beta \neq 0), \label{}
		\end{aligned}
		\right.
		\end{align}
		where the function $I_\nu$ is the modified Bessel function of the first kind:
		\begin{align}
			I_\nu (x) = \sum_{n=0}^{\infty}\frac{1}{n!\Gamma(n + \nu + 1)}\left(\frac{x}{2}\right)^{\nu + 2n} \quad (\nu \in \bR, \ x \in \bR). \label{}
		\end{align}
		We now have
		\begin{align}
		\bP_x[T_0^{(\gamma)} > t] &= \int_{0}^{b}p^{(\gamma)}(t,x,y)dm^{(\gamma)}(y) \label{} \\
		&= \frac{\mathrm{e}^{-\gamma t}}{g_\gamma (x)}\int_{0}^{b}p^{(0)}(t,x,y)g_\gamma(y)dm^{(0)}(y) \label{} \\
		&= \frac{\mathrm{e}^{-\gamma t}}{g_\gamma (x)}\int_{0}^{b}p^{(0)}(t,x,y)dm^{(0)}(y)\int_{0}^{\infty}\mathrm{e}^{-\gamma u}f^{(0)}_y(u)du \label{} \\
		&= \frac{\mathrm{e}^{-\gamma t}}{g_\gamma (x)}\int_{0}^{\infty}\mathrm{e}^{-\gamma u}du
		 \int_{0}^{b}p^{(0)}(t,x,y)f^{(0)}_y(u)dm^{(0)}(y) \label{} \\
		&= \frac{\mathrm{e}^{-\gamma t}}{g_\gamma (x)}\int_{0}^{\infty}\mathrm{e}^{-\gamma u}
		\bP_{x}[\bP_{Y^{(0)}_t}[T^{(0)}_0 \in du]] \label{} \\
		&= \frac{\mathrm{e}^{-\gamma t}}{g_\gamma (x)}\int_{0}^{\infty}\mathrm{e}^{-\gamma u} f_x(u + t)du \label{} \\
		&= \frac{1}{g_\gamma (x)}\int_{t}^{\infty}\mathrm{e}^{-\gamma u} f^{(0)}_x(u)du. \label{} 
		\end{align}
		This shows \eqref{eq77}.
		Then from Proposition \ref{hitting_density} we obtain \eqref{eq76}.
		
		From \cite[p.173]{TakemuraTomisaki:htransform} we have \eqref{eq80}. 
		We show \eqref{eq81}. First we consider the case $\beta > 0$.
		By some computation, we can check that
		\begin{align}
		\psi_{\lambda}(x) = \frac{1}{\alpha}x^{\alpha}M(\lambda/\beta + \alpha,1 + \alpha; \beta x ) \quad (x > 0,\  \lambda \in \bR), \label{eq82}
		\end{align}
		where the function $M$ is Kummer's confluent hypergeometric function:
		\begin{align}
			M(a,b;x) = \sum_{n = 0}^{\infty}\frac{(a)_n x^n}{(b)_n n!} \quad (a,b \in \bR, \ x \in \bR). \label{}
		\end{align}
		We consider the values of $\lambda$ for which the function $\psi_{\lambda}$ is square-integrable. We may assume $\lambda < 0$.
		Since the asymptotic behavior of the function $M$ is given by
		\begin{align}
		M(a,b;x) \sim \frac{\Gamma(b)}{\Gamma(a)} x^{a-b}\mathrm{e}^{x} \quad (x \to \infty) \label{}
		\end{align}
		for $a \neq 0,-1,-2,\cdots$ (see e.g., \cite[p.289]{Specialfunction}), the function $\psi_{\lambda}$ is not square-integrable w.r.t.\ $dm$ when $\lambda/\beta + \alpha \neq 0,-1,-2,\cdots$. When $\lambda/\beta + \alpha = 0,-1,-2,\cdots$, the function $\psi_{\lambda}$ is a polynomial and obviously square-integrable w.r.t.\ $dm$. 
		Note that
		\begin{align}
		M(-n,1 + \alpha; \beta x) = \frac{n!}{(1 + \alpha)_n}L^{(\alpha)}_n(\beta x), \label{}
		\end{align}
		where $L^{(\alpha)}_n(x)$ is the $n$-th Laguerre polynomial of parameter $\alpha$, that is,
		\begin{align}
		L^{(\alpha)}_n (x) = \mathrm{e}^{x}\frac{x^{-\alpha}}{n!}\frac{d^n}{dx^n}(\mathrm{e}^{-x}x^{n + \alpha}) \quad (n \in \bN) \label{}
		\end{align}
		(see e.g., \cite[p.241]{Specialfunction}).
		Since the Laguerre polynomials $\{ L^{(\alpha)}_n(x) \}_n$ comprise the orthogonal basis of $L^2((0,\infty), x^{\alpha}\mathrm{e}^{-x}dx)$, the functions $\{\psi_{-\beta(\alpha + n)}(x)\}$ is so on $L^2((0,\infty),x^{-\alpha}\mathrm{e}^{-\beta x}dx)$.
		Hence the spectral measure only have the point spectrum and the support of $\sigma$ is $\{ \beta (\alpha + n), \  n \geq 0 \}$.
		Since it holds 
		\begin{align}
		\int_{0}^{\infty}L^{(\alpha)}_i(x)L^{(\alpha)}_j(x)x^{\alpha}\mathrm{e}^{-x}dx = \delta_{ij}\frac{\Gamma(i + \alpha+ 1)}{i!} \quad (i,j \in \bN ) \label{}
		\end{align}
		(see e.g., \cite[p.241]{Specialfunction}),
		it follows
		\begin{align}
		\int_{0}^{\infty}\psi_{-\beta(\alpha + n)}(x)^2dm(x) 
		&= \frac{(n!)^2}{\alpha^2\beta^{\alpha+1} \{ (1 + \alpha)_n \}^2}\int_{0}^{\infty}L^{(\alpha)}_n(x)^2x^{\alpha}\mathrm{e}^{-x}dx \label{} \\
		&= \frac{n!\Gamma(\alpha)}{\beta^{\alpha+1}(\alpha)_{n+1}}. \label{}
		\end{align}
		Hence we obtain
		\begin{align}
		\sigma\{ \beta (n + \alpha) \} = \frac{\beta^{\alpha+1}(\alpha)_{n+1}}{n!\Gamma(\alpha)} \quad (n \geq 0). \label{spectralmeasure}
		\end{align}	
		
		Next we show the case $\beta < 0$.
		Let us consider the map
		\begin{align}
			L^2((0,\infty),dm^{(\alpha,-\beta)}) \ni f \longmapsto \mathrm{e}^{\beta x}f \in L^2((0,\infty),dm^{(\alpha,\beta)}). \label{eq83}
		\end{align}
		Obviously, this map is unitary.
		Moreover, since it holds
		\begin{align}
			\cL^{(\alpha,\beta)}(\mathrm{e}^{\beta x}\psi_{\lambda}^{(\alpha,-\beta)}(x)) = (\lambda - \beta (\alpha -1))(\mathrm{e}^{\beta x}\psi_{\lambda}^{(\alpha,-\beta)}(x)) \label{}
		\end{align}
		and
		\begin{align}
			\frac{d}{ds^{(\alpha,\beta)}}(\mathrm{e}^{\beta x}\psi^{(\alpha,-\beta)}_{\lambda}(x)) = \beta x^{1-\alpha}\psi^{(\alpha,-\beta)}_{\lambda}(x) + \mathrm{e}^{\beta x}\frac{d}{ds^{(\alpha,-\beta)}}\psi_{\lambda}^{(\alpha,-\beta)}(x), \label{}
		\end{align}
		 we can see from \eqref{eq82} that
		\begin{align}
			\psi_{\lambda}^{(\alpha,\beta)}(x) = \mathrm{e}^{\beta x}\psi^{(\alpha,-\beta)}_{\lambda + \beta (\alpha - 1)}, \label{}
		\end{align}
		where we denote the function defined in \eqref{eq63} for $\cL^{(\alpha,\beta)}$ by $\psi_{\lambda}^{(\alpha,\beta)}$.
		Then from the unitarity of the map \eqref{eq83} and the argument for the case $\beta > 0$, the functions $\{ \psi_{-\beta(n + 1)}^{(\alpha,\beta)}, n \geq 0 \}$ comprise the orthogonal basis of $L^2((0,\infty), dm^{(\alpha,\beta)})$ and therefore we obtain \eqref{eq81} for $\beta < 0$.
		
		Finally, we show the case $\beta = 0$.
		Note that we can see from some computation that
		\begin{align}
			\psi_{\lambda}(x) = \Gamma(\alpha) \left(\frac{x}{\lambda}\right)^{\alpha/2}I_\alpha(2\sqrt{\lambda x}) \quad (x > 0, \ \lambda \in \bR). \label{}
		\end{align}
		From \eqref{eq20} and \eqref{eq76}, we have
		\begin{align}
			\int_{0}^{\infty}\mathrm{e}^{-\lambda t}\psi_{-\lambda}(x)\sigma^{(0)}(d\lambda)
			= \frac{1}{\Gamma(\alpha)}x^{\alpha}t^{-\alpha-1}\mathrm{e}^{-x/t}.
		\end{align}
		Since it holds that
		\begin{align}
			\frac{d}{dx}(x^\nu I_\nu(x)) = x^\nu I_{\nu-1}(x), \quad I_\nu(x) \sim \frac{\mathrm{e}^{x}}{\sqrt{2\pi x}} \quad (\nu \in  \bR,\ x \to \infty) \label{}
		\end{align}
		(see e.g., \cite[p.67, p.139]{Specialfunction}), we can see
		\begin{align}
			\frac{d}{dx}\int_{0}^{\infty}\mathrm{e}^{-\lambda t}\left|\frac{d}{dx}\psi_{-\lambda}(x)\right|\sigma^{(0)}(d\lambda) < \infty \quad (x > 0). \label{}
		\end{align}
		Thus we have
		\begin{align}
			\int_{0}^{\infty}\mathrm{e}^{-\lambda t}\sigma^{(0)}(d\lambda) 
			=& \left.\frac{d}{ds(x)}\int_{0}^{\infty}\mathrm{e}^{-\lambda t}\psi_{-\lambda}(x)\sigma^{(0)}(d\lambda)\right|_{x = 0} \label{} \\
			=& \left.\frac{d}{ds(x)}\frac{1}{\Gamma(\alpha)}x^{\alpha}t^{-\alpha-1}\mathrm{e}^{-x/t}\right|_{x = 0} \label{} \\
			=& \frac{\alpha t^{-\alpha - 1}}{\Gamma(\alpha)}. \label{}
		\end{align}
		From the uniqueness of the Laplace transform, we obtain \eqref{eq81}.

	\end{proof}
	We give the classification of the boundary $\infty$ for $\cL^{(\gamma)}$:
	\begin{Prop}\label{classification_of_boundary}
		For $\alpha > 0, \ \beta \in \bR,\ \gamma \geq 0$, the boundary $\infty$ for $\cL^{(\gamma)}$ is natural. 
	\end{Prop}
		
	\begin{proof}
		Let $\beta > 0$. From \eqref{eq62} and \eqref{eq50}, we have
		\begin{align}
			s(x) - s(1) &= \int_{1}^{x}y^{\alpha - 1}\mathrm{e}^{\beta y}\frac{dy}{g_\gamma^2(y)} \label{} \\
			&\asymp \int_{1}^{x}y^{\alpha + 2\gamma / \beta- 1}\mathrm{e}^{\beta y}dy \xrightarrow[x \to \infty]{} \infty, \label{}
		\end{align}
		where $f_1 \asymp f_2$ means there exists a constant $c > 0$ such that
		$(1/c)f_1(x) \leq f_2(x) \leq c f_1(x)$ for large $x > 0$.
		Note that from L'H\^opital's rule, it holds for $\delta \in \bR$
		\begin{align}
			\int_{x}^{\infty}y^{\delta}\mathrm{e}^{-\beta y}dy \sim \frac{1}{\beta}x^{\delta}\mathrm{e}^{-\beta x} \quad (x \to \infty). \label{}
		\end{align}
		We have
		\begin{align}
			\int_{1}^{\infty}ds^{(\gamma)}(x)\int_{x}^{\infty}dm^{(\gamma)}(y) 
			&\asymp \int_{1}^{\infty}x^{\alpha + 2\gamma / \beta -1}\mathrm{e}^{-\beta x}dx\int_{x}^{\infty}y^{-\alpha - 2\gamma / \beta}\mathrm{e}^{-\beta y}dy \label{} \\
			&\asymp \int_{1}^{\infty}\frac{dx}{x} = \infty. \label{}  
		\end{align}
		Thus the boundary $\infty$ is natural.
		We can show the cases of $\beta = 0$ and $\beta < 0$ by the similar argument and hence we omit them. 
	\end{proof}

	\section{Convergence to non-minimal quasi-stationary distributions for Kummer diffusions with negative drifts} \label{section:convtoqsd}
	
	Let us apply Theorem \ref{main-theorem-03} to Kummer diffusions with negative drifts and give a sufficient condition on initial distributions under which the conditional process converges to each non-minimal quasi-stationary distribution specified. 
	
	We give a necessary and sufficient condition for Kummer diffusions with negative drifts to satisfy the condition of Theorem \ref{QSD-char01}:
	
	\begin{Prop} \label{NSCforQSDforKum}
		For $\cL^{(\alpha,\beta,\gamma)} \ (\alpha > 0, \beta \in \bR, \ \gamma \geq 0)$,
		the condition of Theorem \ref{QSD-char01} holds if and only if one of the Case 1-3 in \eqref{cases} holds.
	\end{Prop}

	\begin{proof}
		Let $\beta > 0$. Obviously, it holds $m^{(\gamma)}(1,\infty) < \infty$ and $s^{(\gamma)}(\infty) = \infty$.
		From \eqref{eq50}, we have
		\begin{align}
			m^{(\gamma)}(x,\infty) (s^{(\gamma)}(x) - s^{(\gamma)}(1))
			&\asymp (x^{-\alpha -2\gamma /\beta}\mathrm{e}^{-\beta x})(x^{\alpha + 2\gamma /\beta - 1}\mathrm{e}^{\beta x}) \label{} \\
			&\asymp 1/x \xrightarrow[x \to \infty]{} 0. \label{}
		\end{align}
		
		Let $\beta = 0$.
		We can easily check $s^{(\gamma)}(\infty) = \infty$ for $\gamma \geq 0$ and
		\begin{align}
			\lim_{x \to \infty} m^{(0)}(x,\infty)(s^{(0)}(x) - s^{(0)}(1)) = \infty. \label{}   
		\end{align}
		For $\gamma > 0$, from \eqref{eq50} it holds
		\begin{align}
			m^{(\gamma)}(x,\infty)(s^{(\gamma)}(x) - s^{(\gamma)}(1)) 
			\asymp \mathrm{e}^{-4\sqrt{\gamma x}} \cdot \mathrm{e}^{4\sqrt{\gamma x}} = 1. \label{}
		\end{align}

		Let $\beta < 0$. It holds $s^{(0)}(\infty) < \infty$ from \eqref{speedmeasure_scalefunction}. For $\gamma > 0$, we have from \eqref{eq62}
		\begin{align}
			s^{(\gamma)}(x) - s^{(\gamma)}(1) &\asymp \int_{1}^{x}y^{-\alpha - \gamma / \beta }\mathrm{e}^{-\beta y}dy \label{} \\
			&\asymp x^{1 - \alpha - 2\gamma / \beta} \mathrm{e}^{-\beta x} \xrightarrow[x \to \infty]{} \infty. \label{}
		\end{align}
		Similarly, we can show $m(1,x) \asymp x^{-2 + \alpha + 2\gamma / \beta }\mathrm{e}^{\beta x}$ and thus $m(1,\infty) < \infty$. 
		Then it holds
		\begin{align}
			m(x,\infty)(s(x) - s(1)) \asymp 1/x \xrightarrow[x \to \infty]{}. \label{}
		\end{align}		
	\end{proof}
	
	For the process $Y^{(\alpha,\beta,\gamma)}$, the first hitting uniqueness holds on $\cP(0,\infty)$. We show this fact in more general settings as follows:
	
	\begin{Thm}
		Let $X$ be a $\frac{d}{dm}\frac{d}{ds}$-diffusion on $[0,b) \ ( 0 < b \leq \infty)$ and $s(b) = \infty$.
		Suppose the hitting densities $f_x(t)$ of $0$ have the following form
		\begin{align}
		f_x(t) = u(x)w(t)\mathrm{e}^{-v(x)y(t)} \quad (0 < x < b, \ t > 0) \label{eq46}
		\end{align}
		for some strictly positive functions $u(x)$ and $v(x)$ on $(0,b)$ and some strictly positive function $w(t)$ and $y(t)$ on $(0,\infty)$,
		In addition, suppose $v$ is strictly increasing continuous and $y(0,\infty) = (0,\infty)$.
		Then the first hitting uniqueness holds on $\cP(0,\infty)$.
	\end{Thm}
	
	\begin{proof}
		Suppose $\mu_1$ and $\mu_2 \in \cP(I)$ satisfy
		\begin{align}
		\bP_{\mu_1}[T_0 \in dt] = \bP_{\mu_2}[T_0 \in dt] \label{}
		\end{align}
		and set $\mu = \mu_1 - \mu_2$. It holds
		\begin{align}
		\int_{0}^{b}f_x(t)\mu(dx) = 0 \quad (t > 0). \label{eq64}
		\end{align}
		Note that from the continuity of $f_x(t) / w(t)$ w.r.t.\ $t$, the equality \eqref{eq64} holds for every $t > 0$.
		From \eqref{eq46} and by a change of variables, we have
		\begin{align}
		0 = \int_{v(0)}^{v(b)}u(v^{-1}(x))\mathrm{e}^{-xy(t)}\mu(v^{-1}(dx)). \label{}
		\end{align} 
		Since $y(0,\infty) = (0,\infty)$, it holds from the uniqueness of the Laplace transform
		\begin{align}
		u(x)\mu(dx) = 0 \quad \text{on} \ (0,b). \label{}
		\end{align}
		Since $u(x) > 0$, we obtain the desired result. 
	\end{proof}
	
	Now we go on to the proof of Theorem \ref{thm:doaofKum}.
	For the proof of (i) of Theorem \ref{thm:doaofKum}, we need the following lemma, which enables us to cut off the integral region for the asymptotic behavior of the Laplace transform:
	\begin{Lem}\label{laplaceprinciple}
		Let $f: (0,\infty) \to [0,\infty)$ and assume
		\begin{align}
			\log f(x) \sim \delta \sqrt{x} \quad (x \to \infty) \label{eq65}
		\end{align}
		for $\delta > 0$ and
		\begin{align}
			\int_{0}^{\infty}\mathrm{e}^{-x/t}f(x)dx < \infty  \quad (t > 0). \label{}
		\end{align}
		Then for every $\eps > 0$, it holds
		\begin{align}
			\log \int_{0}^{\infty}\mathrm{e}^{-x/t}f(x)dx \sim \frac{\delta^2}{4}t \label{}
		\end{align}
		and
		\begin{align}
			\int_{0}^{\infty}\mathrm{e}^{-x/t}f(x)dx \sim \int_{(\delta^2/4 - \eps)t^2}^{(\delta^2/4 + \eps)t^2}\mathrm{e}^{-x/t}f(x)dx \quad (t \to \infty). \label{}
		\end{align}
	\end{Lem}

	\begin{proof}
		Since it holds
		\begin{align}
			\lim_{t \to \infty}\int_{0}^{1}\mathrm{e}^{-x/t}f(x)dx < \infty \quad \text{and} \quad \lim_{t \to \infty}\int_{1}^{\infty}\mathrm{e}^{-x/t}f(x)dx = \infty, \label{}
		\end{align}
		we may assume without loss of generality that $f(x) = 0$ for $0 < x < 1$.
		It is enough to show
		\begin{align}
			\lim_{t \to \infty}\frac{\int_{1}^{(\delta^2/4 - \eps)t}\mathrm{e}^{-x/t}f(x)dx}{\int_{1}^{\infty}\mathrm{e}^{-x/t}f(x)dx} = 0 \label{eq66}
		\end{align}
		and
		\begin{align}
			\lim_{t \to \infty}\frac{\int_{(\delta^2/4 + \eps)t}^{\infty}\mathrm{e}^{-x/t}f(x)dx}{\int_{1}^{\infty}\mathrm{e}^{-x/t}f(x)dx} = 0. \label{eq67}
		\end{align}
		Let 
		\begin{align}
			h(x) = \frac{\log (x^2f(x))}{\sqrt{x}} - \delta \quad (x > 1). \label{}
		\end{align}
		Then from \eqref{eq65}, we have $\lim_{x \to \infty}h(x) = 0$.
		It follows
		\begin{align}
			\int_{1}^{\infty}\mathrm{e}^{-x/t}f(x)dx = \int_{1}^{\infty}\mathrm{e}^{-\varphi_t(x)}\frac{dx}{x^2} \label{}
		\end{align}
		where
		\begin{align}
			\varphi_t(x) = x/t - (\delta + h(x))\sqrt{x}  
		\end{align}
		Note that
		\begin{align}
			\varphi_t(x) = \frac{1}{t}\left( \sqrt{x} - \frac{\delta + h(x)}{2}t \right)^2 - \frac{(\delta + h(x))^2}{4}t. \label{}
		\end{align}
		Let $\theta := \delta / 2  - \sqrt{\delta^2 / 4 - \eps} > 0$ and take $R > 1$ so that
		\begin{align}
			|h(x)| < \theta \quad \text{and} \quad \frac{2\delta |h (x)| + h(x)^2}{4} < \theta^2 / 8 \quad (x > R). 
		\end{align}
		Then for $R < x < (\delta^2/4 - \eps) t^2$, it follows
		\begin{align}
			\frac{\delta + h(x)}{2}t - \sqrt{x} 
			> \frac{\delta + h(x)}{2}t - t\sqrt{\delta^2/4 - \eps}			
			> \frac{\theta}{2}t \label{}
		\end{align}
		and thus
		\begin{align}
			\varphi_t(x) \geq (\theta^2 / 8 - \delta^2 / 4)t. \label{}
		\end{align}
		Then it follows
		\begin{align}
			\int_{R}^{(\delta^2/4 - \eps) t^2}\mathrm{e}^{-\varphi_t(x)}\frac{dx}{x^2}
			&\leq \mathrm{e}^{( \delta^2 /4 - \theta^2 / 8)t}\int_{R}^{(\delta^2/4 - \eps) t^2}\frac{dx}{x^2} \label{} \\
			&\leq \mathrm{e}^{(\delta^2 /4 - \theta^2 / 8 )t}. \label{}
		\end{align}
		For showing \eqref{eq66}, it is hence enough to show 
		\begin{align}
			\log \int_{1}^{\infty}\mathrm{e}^{-x/t}f(x)dx \sim \frac{\delta^2}{4}t \quad (t \to \infty). \label{eq68}
		\end{align}
		From \cite[Theorem 4.12.10 (ii)]{Regularvariation}, it holds
		\begin{align}
			\log \int_{0}^{x}f(y)dy \sim \delta \sqrt{x} \quad (x \to \infty). \label{}
		\end{align}
		From Kohlbecker's Tauberian Theorem \cite[Theorem 4.12.1]{Regularvariation}, we obtain therefore \eqref{eq68}.
		We can show \eqref{eq67} by a similar argument.
	\end{proof}
	
	Now we proceed to the proof of Theorem \ref{thm:doaofKum}.
	\begin{proof}[Proof of Theorem \ref{thm:doaofKum}]
		First we show (i). From Proposition \ref{non-minimalQSD_density} and Theorem \ref{hitting-density-example}, it is enough to show
		\begin{align}
			\lim_{t \to \infty}\frac{d}{dt} \log \int_{0}^{\infty}\mathrm{e}^{-x/t}\frac{x^{\alpha/2}}{K_\alpha (2\sqrt{\gamma x})}\mu(dx) = \delta^2/4. \label{}
		\end{align}
		From \eqref{eq50}, it holds
		\begin{align}
			\log \tilde{\rho}(x) := \log \frac{x^{\alpha/2}\rho(x)}{K_\alpha (2\sqrt{\gamma x})} \sim \delta \sqrt{x} \quad (x \to \infty). \label{eq70}
		\end{align}
		Take $\eps > 0$. Since it holds
		\begin{align}
			\frac{d}{dt} \log \int_{0}^{\infty}\mathrm{e}^{-x/t}\frac{x^{\alpha/2}}{K_\alpha (2\sqrt{\gamma x})}\mu(dx)
			= \frac{\int_{0}^{\infty}\mathrm{e}^{-x/t}x\tilde{\rho}(x)dx}{t^2 \int_{0}^{\infty}\mathrm{e}^{-x/t}\tilde{\rho}(x)dx}, \label{}
		\end{align}
		we have from \eqref{eq70} and Lemma \ref{laplaceprinciple}
		\begin{align}
			\frac{\int_{0}^{\infty}\mathrm{e}^{-x/t}x\tilde{\rho}(x)dx}{t^2 \int_{0}^{\infty}\mathrm{e}^{-x/t}\tilde{\rho}(x)dx}
			\sim \frac{\int_{(\delta^2/4 - \eps)t^2}^{(\delta^2/4 + \eps)t^2}\mathrm{e}^{-x/t}x\tilde{\rho}(x)dx}{t^2 \int_{(\delta^2/4 - \eps)t^2}^{(\delta^2/4 + \eps)t^2}\mathrm{e}^{-x/t}\tilde{\rho}(x)dx} \label{}
		\end{align}
		and obviously we have
		\begin{align}
			\int_{(\delta^2/4 - \eps)t^2}^{(\delta^2/4 + \eps)t^2}\mathrm{e}^{-x/t}x\tilde{\rho}(x)dx 
			\lesseqgtr (\delta^2/4 \pm \eps)t^2 \int_{(\delta^2/4 - \eps)t^2}^{(\delta^2/4 + \eps)t^2}\mathrm{e}^{-x/t}\tilde{\rho}(x)dx. \label{}
		\end{align}
		Since $\eps > 0$ can be taken arbitrary small, we obtain
		\begin{align}
			\frac{\int_{0}^{\infty}\mathrm{e}^{-x/t}x\tilde{\rho}(x)dx}{t^2 \int_{0}^{\infty}\mathrm{e}^{-x/t}\tilde{\rho}(x)dx} \xrightarrow{t \to \infty} \delta^2/4. \label{}
		\end{align}

		Next we show (ii).
		From the proof of Proposition \ref{non-minimalQSD_density}, it is enough to show that
		the function $f_\mu(\log t)$ varies regularly at $\infty$ with exponent $-\lambda$.
		From Theorem \ref{hitting-density-example}, we have
		\begin{align}
			f_\mu(\log t) = \frac{1}{\Gamma(\alpha)}t^{\beta -\gamma}h(t)^{1 + \alpha}\int_{0}^{\infty}\frac{x^\alpha}{g_\gamma(x)}\mathrm{e}^{-h(t)x}\mu(dx), \label{}
		\end{align}
		where
		\begin{align}
			h(t) = \frac{\beta}{t^{\beta} - 1} \quad (t > 1). \label{}
		\end{align}
		The inverse function $h^{-1}$ of $h$ is given as
		\begin{align}
			h^{-1}(s) = \left(1 + \frac{\beta}{s} \right)^{1/\beta} \quad (s > 0). \label{}
		\end{align}
		Note that the function $h^{-1}(s)$ varies regularly at $s = 0$ with exponent $-1/\beta$.
		By considering the function $f(\log h^{-1}(s))$, it follows that the function $f_\mu(\log t)$ varies regularly at $t = \infty$ with exponent $-\lambda$
		if and only if the function
		\begin{align}
			\int_{0}^{\infty}\frac{x^\alpha}{g_\gamma(x)}\mathrm{e}^{-sx}\mu(dx) \label{eq69}
		\end{align}
		varies regularly at $s = 0$ with exponent $-\alpha - (\gamma -\lambda)/\beta$.
		From Karamata's Tauberian Theorem \cite[Theorem 1.7.1]{Regularvariation}, it is equivalent to that the function
		\begin{align}
			\int_{0}^{x}\frac{y^{\alpha}}{g_\gamma(y)}\mu(dy)
		\end{align}
		varies regularly at $x = \infty$ with exponent $\alpha + (\gamma - \lambda) / \beta$.
		Then from \eqref{eq50} and \cite[Theorem 1.6.4]{Regularvariation}, it is equivalent to the function $\mu(x,\infty)$ varies regularly at $x = \infty$ with exponent $-\lambda / \beta$, and therefore we obtain (ii).
		
		Finally, we show (iii).
		The proof of this case is quite similar to that of (ii).
		From Theorem \ref{hitting-density-example}, we have
		\begin{align}
		f_\mu(\log t) = \frac{1}{\Gamma(\alpha)}t^{\beta -\gamma}h(t)^{1 + \alpha}\int_{0}^{\infty}\frac{x^\alpha}{g_\gamma(x)}\mathrm{e}^{-h(t)x}\mu(dx). \label{}
		\end{align}
		Note that for $\beta < 0$, it holds $\lim_{t \to \infty}h(t) = -\beta$.
		Then the function $f_\mu(\log t)$ varies regularly at $t = \infty$ with exponent $-\lambda$ if and only if the function 
		\begin{align}
			\int_{0}^{\infty}\frac{x^\alpha}{g_\gamma(x)}\mathrm{e}^{-h(t)x}\mu(dx)
			= \frac{(-\beta)^{-\alpha}\Gamma(\alpha)}{\Gamma(1 - \gamma / \beta)}\int_{0}^{\infty}\frac{\mathrm{e}^{-(h(t) + \beta)x}}{U(1-\gamma/\beta, \alpha + 1; -\beta x)}\mu(dx) \label{eq84}
		\end{align}
		varies regularly at $t = \infty$ with exponent $-\lambda -\beta + \gamma$.
		Note that the function $h^{-1}(s)$ varies regularly at $s = -\beta + 0$ with exponent $-1/\beta$. Thus, by denoting $u = s + \beta$, the regular variation at $t = \infty$ of \eqref{eq84} with exponent $-\lambda - \beta + \gamma$ is equivalent to that at $u = 0$ of 
		\begin{align}
			\int_{0}^{\infty}\frac{\mathrm{e}^{-ux}}{U(1-\gamma/\beta, \alpha + 1; -\beta x)}\mu(dx) \label{}
		\end{align}
		with exponent $1 + (\lambda - \gamma) / \beta$.
		Using \eqref{eq50}, the rest of the proof can be made by the same argument in (ii) and hence we omit it.
		The proof is complete.
	\end{proof}

	\appendix

	\bibliography{Bibliography.bib} 
	\bibliographystyle{plain}

\end{document}